\documentclass[12pt]{amsart}
\title[QUASISYMMETRIC EMBEDDING OF THE INTEGER SET...]{QUASISYMMETRIC EMBEDDING OF THE INTEGER SET AND ITS QUASICONFORMAL EXTENSION}
\author[H. FUJINO]{HIROKI FUJINO}
\address{Graduate~School~of~Mathematics, Nagoya~University, Furo-cho Chikusa-ku Nagoya 464-8602, Japan}
\email{m12040w@math.nagoya-u.ac.jp}
\date{\today}
\usepackage[dvips]{graphicx}
\usepackage{amsmath,amssymb}
\usepackage{amsrefs}
\usepackage{ascmac}
\usepackage{mathrsfs}

\makeatletter    
  \def\@seccntformat#1{\csname the#1\endcsname.\quad}
\makeatother

\allowdisplaybreaks

\newtheorem{defthm}{Theorem}[section]
\newtheorem{alpthm}{Theorem}   %\renewcommand{\thealpthm}{\Alph{section}} %chapを変更してうまく合わせる
\newtheorem{empthm}{Theorem}   

\newcounter{claim}{}
\newtheorem{defclaim}{Claim}[claim] 

\newtheorem{defprop}[defthm]{Proposition}

\newtheorem{defdef}[defthm]{Definition}
\newtheorem{deflem}[defthm]{Lemma}
\newtheorem{defex}[defthm]{Example}
\newtheorem{defremark}[defthm]{Remark}

\newcommand{\re}{{\rm Re}}
\newcommand{\im}{{\rm Im}}

\begin{document}

% アブストラクト＋タイトルーーーーーーーーーーーーーーーーー
\begin{abstract}
	We prove that an injection from the integer set into the real line
	admits a quasiconformal extension to the complex plane if and only if
	it is quasisymmetric.
\end{abstract}

\maketitle　%abstractの後に置かないと間隔がおかしくなるようだ。
%ーーーーーーーーーーーーーーーーーーーーーーーーーーーーーー

% 本文  第一節〜ーーーーーーーーーーーーーーーーーーーーーーーー

%%%%%%%%%%%%%%%%%%%%%%%%%%%%%%%%%%%%%%%%%%%%%%%%%%%%%%%%%%%%%%%%%%%%%  CHAPTER 1 %%%%%%%%%%%%%%　
\section{Introduction}   %amsartでは自動的に大文字になる
\label{introduction}
    
Let $E\subset \mathbb{C}$ be a discrete subset. 
In \cite{fujino2}, to study the Teichm\"uller space of
the punctured plane $\mathbb{C}\setminus \mathbb{Z}$,
the author gave some criteria for
$\mathbb{C}\setminus E$ to be quasiconformally equivalent to
$\mathbb{C}\setminus \mathbb{Z}$ (that is, there exists a quasiconformal
mapping $F:\mathbb{C}\rightarrow \mathbb{C}$ such that $F(E)=\mathbb{Z}$).
In this paper, furthermore, we investigate the correspondences between
$E$ and $\mathbb{Z}$ which are the restrictions of global quasiconformal mappings 
$F:\mathbb{C}\rightarrow \mathbb{C}$ such that $F(E)=\mathbb{Z}$.
A motivation of this attempt is to
study the Teichm\"uller modular group of $\mathbb{C}\setminus \mathbb{Z}$
and its action.

Let $\eta:[0,+\infty )\rightarrow [0,+\infty )$ be a homeomorphism and
$f:X\rightarrow \mathbb{R}^n$ be an $\eta-$quasisymmetric embedding from
a subset $X\subset \mathbb{R}^n$ into $\mathbb{R}^n$.
The theory of the quasisymmetry and its quasiconformal extension originated
from the well known study for $X=\mathbb{R}$ and $n=1$ 
by Beurling--Ahlfors \cite{beurling1}.
They proved that a homeomorphism $f:\mathbb{R}\rightarrow
\mathbb{R}$ admits a quasiconformal extension $F:\mathbb{C}\rightarrow
\mathbb{C}$ if and only if $f$ is quasisymmetric.
This result enables us to treat the universal Teichm\"uller space, the Teichm\"uller
space of the unit disk, as the space of all orientation preserving 
quasisymmetric homeomorphisms of the unit circle which fix given three points.
Later, V\"ais\"al\"a posed the following question in \cite[Question 8]{vaisala4}
which is still open;
can $f$ be extended to a $K-$quasiconformal mapping $F:\mathbb{R}^{2n}
\rightarrow \mathbb{R}^{2n}$ with a constant $K=K(n,\eta)\geq 1$
which depends only on $n$ and $\eta$?

For example, Alestalo--V\"ais\"al\"a showed 
that if $f:X\rightarrow \mathbb{R}^n$ is $M-$ biLipschitz, then there always 
exists a $\sqrt{7}M^2-$ biLipschitz extension $F:\mathbb{R}^{2n}\rightarrow
\mathbb{R}^{2n}$ of $f$ (see \cite[Theorem 5.5]{alestalo1}). 
On the other hand, for quasisymmetric embeddings, there is an obstacle; 
Trotsenko--V\"ais\"al\"a proved in \cite[Theorem 6.6]{trotsenko1} that
if $X\subset \mathbb{R}^n$ is not relatively connected, then 
there exists a quasisymmetric embedding $f:X\rightarrow \mathbb{R}^n$
which cannot be extended to a quasisymmetric embedding 
$F:\mathbb{R}^n\rightarrow\mathbb{R}^N$ for any $N\geq n$.
Since global quasiconformal mappings $F:\mathbb{R}^{2n}\rightarrow
\mathbb{R}^{2n}$ are also quasisymmetric (see \cite[Theorem 11.14]{heinonen1}),
this fact implies that the V\"ais\"al\"a problem cannot be solved
affirmatively for general subsets $X$ even if $n=1$.

%%%Vellis' results%%%%%%

According to the recent study by Vellis \cite{vellis1}, he showed that
if $X\subset \mathbb{R}$ is $M-$relatively connected, then 
every $\eta-$quasisymmetric embedding $f:X\rightarrow \mathbb{R}^n$
can be extended to an $\eta'-$quasisymmetric embedding $F:\mathbb{R}
\rightarrow \mathbb{R}^N$, where $\eta'$ depends only on $\eta$ 
and $M$, and $N(\geq n)$ depends only on $n$, $\eta$, and $M$.
Considering the one dimensional case of the V\"ais\"al\"a problem,
it is interesting to find out whether we can choose $N=2$ uniformly
when $n=1$ in the Vellis's result.\\

%%%Main Results %%%%%%%%%

Let us consider the case of $X=\mathbb{Z}$ and $n=1$. 
In this paper, we would like to give 
detailed observations on quasisymmetric embeddings $f:\mathbb{Z}\rightarrow
\mathbb{R}$, as an example of a relatively connected set
 for which the V\"ais\"al\"a problem can be solved affirmatively;

\setcounter{alpthm}{0}
\begin{alpthm}{\rm (Extensibility of quasisymmetric embeddings of
	$\mathbb{Z}$)} \label{T5.1}\\
	Every $\eta-$quasisymmetric embedding $f:\mathbb{Z}\rightarrow 
	\mathbb{R}$ admits a $K=K(\eta)$ $-$quasiconformal extension
	$F:\mathbb{C}\rightarrow \mathbb{C}$ where
	$K=K(\eta)$ is a constant depending only on $\eta$.
\end{alpthm}

Compared to the Beurling--Ahlfors extension theorem, the difficulty
in our case is that $f$ can change the magnitude relation.
To prove Theorem \ref{T5.1}, first, we will observe the extensibility
of quasisymmetric automorphisms $f:\mathbb{Z}\rightarrow \mathbb{Z}$
in Section \ref{automorphism} and \ref{proofT1}.
\setcounter{alpthm}{1}
\begin{alpthm}{\rm (Extensibility of quasisymmetric automorphisms of $\mathbb{Z}$)}
\label{T3.2}
	For a bijection $f:\mathbb{Z}\rightarrow \mathbb{Z}$, the following
	conditions are quantitatively equivalent;
	\begin{enumerate}
		\item $f$ is $\eta-$quasisymmetric.
		\item $\{a_n:=f(n)\}_{n\in\mathbb{Z}}$ satisfies the $\lambda-$three point condition.
		\item $f$ admits a $K-$quasiconformal extension $F:
			\mathbb{C}\rightarrow \mathbb{C}$.
	\end{enumerate}
\end{alpthm}
\noindent
We say that a sequence $\{a_n\}_{n\in \mathbb{Z}}$ satisfies the $\lambda-$three point
condition for $\lambda \geq 1$ if $|a_n-a_m|/|a_n-a_k|\leq \lambda$ holds
for any integers $n<m<k$. Thus Theorem \ref{T3.2} does not only state
every quasisymmetric automorphism of $\mathbb{Z}$ is
quasiconformally extensible, but also characterizes the quasisymmetry by a 
simple geometric condition. Further, an analogous theorem holds for
quasisymmetric automorphisms of $E=\{e^n\}_{n\in \mathbb{Z}}$ 
(see Theorem \ref{T3.3}).

Next, we will observe a subset of $\mathbb{R}$ which is an image of a
quasisymmetric embedding $f:\mathbb{Z}\rightarrow \mathbb{R}$ 
in Section \ref{images}, to complete
the proof of Theorem \ref{T5.1}. In this case, such subsets can also be
characterized by a simple geometric condition as follows;

\setcounter{alpthm}{2}
\begin{alpthm}{\rm (Characterizetion of quasisymmetric images)} \label{T4.1}
	For a subset $E\subset \mathbb{R}$, the following conditions
	are quantitatively equivalent;
	\begin{enumerate}
		\item There exists an $\eta-$quasisymmetric bijection 
			$f:\mathbb{Z}\rightarrow E$.
		\item $E$ can be written as a monotone increasing sequence $E=\{a_n\}_{n\in \mathbb{Z}}$
			with $a_n\rightarrow\pm \infty\ (n\rightarrow\pm \infty)$, and
			there exists a constant $M\geq 1$ such that the following inequality holds 
			for all $n\in \mathbb{Z}$ and $k\in \mathbb{N}$;
			\[
				\frac{1}{M} \leq \frac{a_{n+k}-a_n}{a_n-a_{n-k}} \leq M.
			\]
		\item There exists a $K-$quasiconformal mapping $F:\mathbb{C}\rightarrow\mathbb{C}$,
			such that $F(\mathbb{Z})=E$.
	\end{enumerate}
\end{alpthm}

    \medskip
%%%%%%%%%%%%%%%%%%%%%%%%%%%%%%%%%%%%%%%%%%%%%%%%%%%%%%%%%%%%%%%%%%%%%  CHAPTER 2 %%%%%%%%%%%%%%
\section{Definitions and Basic properties}
    
First, let $\eta:[0,+\infty )\rightarrow [0,+\infty )$ be a homeomorphism
and $X\subset \mathbb{C}$ be a subset.
An injection $f:X\rightarrow \mathbb{C}$ is said to be 
$\eta-$quasisymmetric if the following inequality holds for
any three points $x,y,z\in X\ (x\neq z)$;
\begin{align}
	\left|\frac{f(x)-f(y)}{f(x)-f(z)}\right|\leq 
	\eta\left(\left|\frac{x-y}{x-z}\right|\right).  \tag{QS} \label{qs}
\end{align}
If $x\neq y$, replacing $y$ and $z$, the following lower estimate holds;
\[
	\left|\frac{f(x)-f(y)}{f(x)-f(z)}\right|\geq 
	\eta\left(\left|\frac{x-y}{x-z}\right|^{-1}\right)^{-1}.
\]
Notice that if there exists at least one $\eta-$quasisymmetric mapping
(and $X$ contains at least two elements), 
applying (\ref{qs}) to $y=z$, 
it turns out that $\eta$ must satisfy $\eta(1)\geq 1$.

Next, let $K\geq 1$ and $\Omega \subset \mathbb{C}$ be a domain.
An orientation preserving homeomorphism $f:\Omega \rightarrow \mathbb{C}$
into $\mathbb{C}$ is said to be $K-$quasiconformal if its distributional
derivatives $f_z$ and $f_{\bar{z}}$ are in the locally integrable class, and 
\[
	K(f):=\underset{z\in\Omega}{{\rm ess.sup}}
		\frac{|f_z(z)|+|f_{\bar{z}}(z)|}{|f_z(z)|-|f_{\bar{z}}(z)|}\leq K.
\]
$K(f)$ is called the maximal dilatation of $f$. 
%We would like to remark that we assume quasiconformal mappings 
%preserve orientation in this paper.

These two concepts are closely related by the so-called egg-yolk principle
(see \cite[Theorem 11.14]{heinonen1}). In particular, for orientation
preserving homeomorphisms from $\mathbb{C}$ onto itself, the quasiconformality
and the quasisymmetry are quantitatively equivalent.

    \medskip
%%%%%%%%%%%%%%%%%%%%%%%%%%%%%%%%%%%%%%%%%%%%%%%%%%%%%%%%%%%%%%%%%%%%%  CHAPTER 3 %%%%%%%%%%%%%%
\section{Key Observation}
\label{observation}

We would like to start from a simple observation which is trivial for ones
who are familiar with quasiconformal mappings. However, this observation
will play a central role in the construction of quasiconformal 
extensions in later sections. \\

Let us consider a rectangle $R_{a,b}=\left\{ z \in \mathbb{C};\ 
|\re z|<a,\ |\im z|<b\right\}$ for $a,b>0$.
For a real number $c\in (0,a)$, we set
\[
f(z):=\left\{
\begin{array}{ll}
z & (z\not \in R_{a,b})\\
\displaystyle \left(a+c-\frac{2c}{b}|y|\right)\frac{x+a}{a-c}-a+iy &
(z=x+iy\in R_{a,b},\ -a\leq x\leq -c)\\
\displaystyle \left(a-c+\frac{2c}{b}|y|\right)\frac{x-a}{a+c}+a+iy &
(z=x+iy\in R_{a,b},\ -c< x\leq a).
\end{array}
\right.
\]
\begin{figure}[h]
       \begin{center}
           \includegraphics[width=12cm]{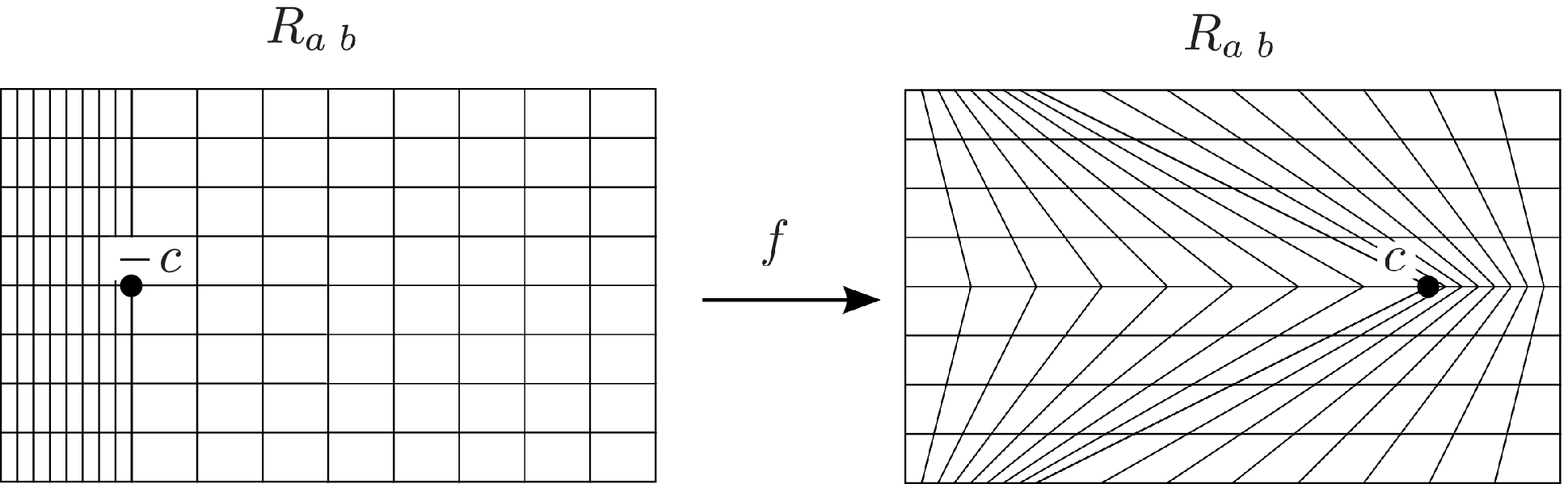} %ファイル名.拡張子を入れる.\ 
       \end{center}
       \caption{}   \label{obs1}
\end{figure}

Then $f$ defines a quasiconformal homeomorphism of $\mathbb{C}$
(see Figure \ref{obs1}).
In particuler $f=id$ on $\mathbb{C}\setminus R_{a,b}$, 
$f(-c)=c$, and its maximal dilatation depends only on $a,b$ and $c$. 
By using this flexible deformation, we have the following preliminary lemma.

\begin{deflem}\label{L1}
Let $n\in \mathbb{N}$ and $\delta \in (0,+\infty]$, and 
let 
\[
R(n,\delta)=\left\{z\in \mathbb{C};\ 
\frac{1}{2}<\re z< n+\frac{1}{2},\ |\im z|< \delta\right\}.
\]
Then, for any bijection
$f:\{1,2,\cdots,n\}\rightarrow \{1,2,\cdots,n\}$,
there exists a $K=K(n,\delta)-$quasiconformal extension
$\widetilde{f}:\mathbb{C}\rightarrow \mathbb{C}$ of f,
such that $\widetilde{f}=id$ on $\mathbb{C}\setminus R(n,\delta)$,
where $K=K(n,\delta)$ is a constant depending only on $n$ and $\delta$.
\end{deflem}

\begin{proof}
We prove the claim by induction with respect to $n\in \mathbb{N}$.
Clearly, the claim holds for $n=1$. We assume the claim holds
for $n-1\ (n\geq 2)$. 

Let $f:\{1,2,\cdots,n\}\rightarrow\{1,2,\cdots,n\}$ 
be a bijection, and let $m:=f(n)$. By the preceding observation,
%we can easily construct a global quasiconformal mapping $g:\mathbb{C}
%\rightarrow \mathbb{C}$ which permutes $n$ and $m$,
we can easily construct a quasiconformal mapping $g_1:\mathbb{C}
\rightarrow \mathbb{C}$ which maps $m\mapsto m-i\delta/2,\ n\mapsto n+i\delta/2$,
fixes the other integers, and is identity on $\mathbb{C}\setminus 
R(n,\delta)$ (see Figure \ref{obs2}).
Similarly we construct global quasiconformal mappings
$g_2$ which maps $m-i\delta/2\mapsto n-i\delta/2,\ n+i\delta/2\mapsto m+i\delta/2$, 
and $g_3$ which maps $n-i\delta/2 \mapsto n,\ m+i\delta/2 \mapsto m$.
Then $g:=g_3 \circ g_2 \circ g_1$ is a quasiconformal mapping which permutes
$n$ and $m$, fixes the other integers, and is identity on $\mathbb{C}
\setminus R(n,\delta)$.  
Since the possible values of $m$ are only $n-$kinds, the maximal dilatation of $g$
is bounded by a constant depending only on $n$ and $\delta$.
\begin{figure}[h]		
       \begin{center}
           \includegraphics[width=13cm]{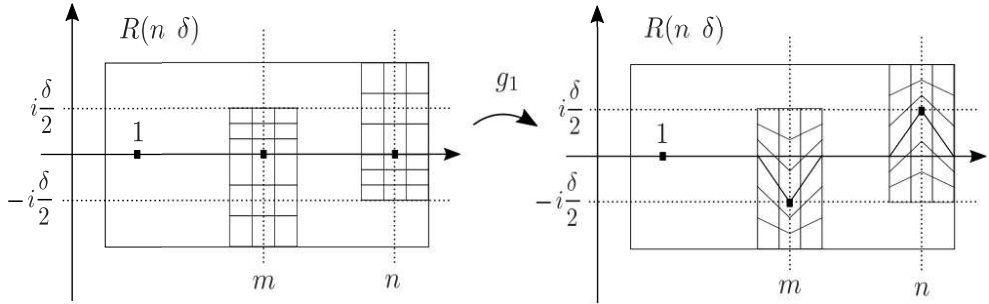} %ファイル名.拡張子を入れる.\ 
       \end{center}
       \caption{:　construction of $g_1$.}   \label{obs2}
\end{figure}

By the construction, $f_1:=g^{-1}\circ f$ fixes $n$. 
Thus $f_1$ defines a permutation of the set 
$\{1,2,\cdots,n-1\}$. From the assumption, $f_1$ extends to a
global quasiconformal mapping which satisfies the conditions of the claim.
Thus, we have a desired extension $\widetilde{f}=g\circ f_1$.
\end{proof}

Each bijection $f:\mathbb{Z}\rightarrow \mathbb{Z}$ can be regarded as 
a permutation of countably many elements. On the other hand, Lemma \ref{L1}
guarantees that any permutation of $\{1,2,\cdots,n\}\subset \mathbb{C}$, a finite set,
can be represented by a quasiconformal homeomorphism of $\mathbb{C}$ which deforms only
a small neighborhood $R(n,\delta)$ of $\{1,2,\cdots,n\}$.

    \medskip 
%%%%%%%%%%%%%%%%%%%%%%%%%%%%%%%%%%%%%%%%%%%%%%%%%%%%%%%%%%%%%%%%%%%%%  CHAPTER 4 %%%%%%%%%%%%%%
\section{Extensibility of quasisymmetric automorphisms of $\mathbb{Z}$}
\label{automorphism}

The aim of this section is to prove Theorem \ref{T3.2}.
For this purpose, it is useful to consider bijections
$f:\mathbb{Z}\rightarrow \mathbb{Z}$ as sequences.

\subsection{Splittable bijective sequence}\ 

We say that a sequence $A=\{a_n\}_{n\in\mathbb{Z}}= \mathbb{Z}$ is \emph{bijective}
if the correspondence $n\mapsto a_n\ (\mathbb{Z}\rightarrow \mathbb{Z})$
is bijective. For two integers $k,\ell \in\mathbb{Z}\ (k\leq \ell)$, we use
the following notations;
\begin{eqnarray*}
[k,\ell]_{\mathbb{Z}}&:=&[k,\ell]\cap \mathbb{Z}=\{k,k+1,\cdots ,\ell-1,\ell \},\\
\bigl| [k,\ell]_{\mathbb{Z}}\bigr|&:=&\# [k,\ell]_{\mathbb{Z}}\ \ =\ \ell-k+1.
\end{eqnarray*}
Remark that we allow the case of $k=\ell$ in the above notations, and in this case,
$[k,k]_{\mathbb{Z}}=\{k\}$ and $\bigl| [k,k]_{\mathbb{Z}}\bigr|=1$.

\begin{defdef}\label{D3.1}
Let $A=\{a_n\}_{n\in\mathbb{Z}}=\mathbb{Z}$ be bijective.
We say that an interval $I=[k,\ell]_{\mathbb{Z}}$ splits $A$ if 
$\displaystyle a_n>\max_{j\in I} a_j$ holds for all $n>\ell$,
and $\displaystyle a_n<\min_{j\in I} a_j$ holds for all $n<k$.

Further, we say that $A$ is $C-$splittable for a constant $C\geq 1$ if
there exists a strictly monotone increasing sequence $\{k_n\}_{n\in\mathbb{Z}}
\subset \mathbb{Z}$ which satisfies the following conditions
for all $n\in \mathbb{Z}$;
\begin{itemize} \setlength{\itemsep}{1ex}
\item the interval $I_n:=[k_n +1,k_{n+1}]_{\mathbb{Z}}$ splits $A$,
\item $|I_n|=k_{n+1}-k_n\leq C$.
\end{itemize}

\end{defdef}

By the definition, the following is immediately confirmed. 

\begin{deflem} \label{L2}
If an interval $I=[k,\ell]_{\mathbb{Z}}$ splits a 
bijective sequence $A=\{a_n\}_{n\in \mathbb{Z}}= \mathbb{Z}$,
then $\{a_n\}_{n\in I}$ is an interval $I'=[k',\ell ']_{\mathbb{Z}}$ with
$|I|=|I'|$. Further $\{a_n\}_{\mathbb{Z}_{>\ell}}=[\ell '+1,+\infty)_{\mathbb{Z}}
=\{\ell'+1,\ell'+2,\cdots\}$,
and $\{a_n\}_{n\in\mathbb{Z}_{<k}}=(-\infty,k'-1]_{\mathbb{Z}}
=\{k'-1,k'-2,\cdots\}$.
\end{deflem}

\begin{proof}
Let $k'= \min_{j\in I} a_j$ and $\ell'=\max_{j\in I} a_j$. By the definition,
we have $\{a_n\}_{n\in I}\subset [k',\ell']_{\mathbb{Z}}$,\ \  
$\{a_n\}_{\mathbb{Z}_{>\ell}}\subset [\ell '+1,+\infty)_{\mathbb{Z}}$,
and $\{a_n\}_{n\in\mathbb{Z}_{<k}}\subset (-\infty,k'-1]_{\mathbb{Z}}$.
Since $A=\{a_n\}_{n\in \mathbb{Z}}$ is bijective, the above implications
must be equalities. Further $|I'|=\# \{a_n\}_{n\in I}=|I|$.
\end{proof}

%This Lemma \ref{L2} claims that the interval $I$ splits the correspondence $n\mapsto a_n$
%into three permutations of intervals.
Note that if an interval $I$ splits a bijective sequence 
$A=\{a_n\}_{n\in \mathbb{Z}}=\mathbb{Z}$, 
then $a_n\rightarrow \pm \infty$ as
$n\rightarrow \pm \infty$.

\begin{defex} 
Define a bijective sequence $A=\{a_n\}_{n\in\mathbb{Z}}$ by
\begin{eqnarray*}
a_{6n}=6n,&\ a_{6n+1}=6n+2,&\ a_{6n+2}=6n-2,\\
a_{6n+3}=6n+3,&\ a_{6n+4}=6n+5,&\ a_{6n+5}=6n+1.
\end{eqnarray*}
Then $a_n\rightarrow \pm \infty$ as $n\rightarrow \pm \infty$, but there is no
interval which splits $A$.
\end{defex}
\begin{figure}[h]
       \begin{center}
           \includegraphics[width=10.5cm]{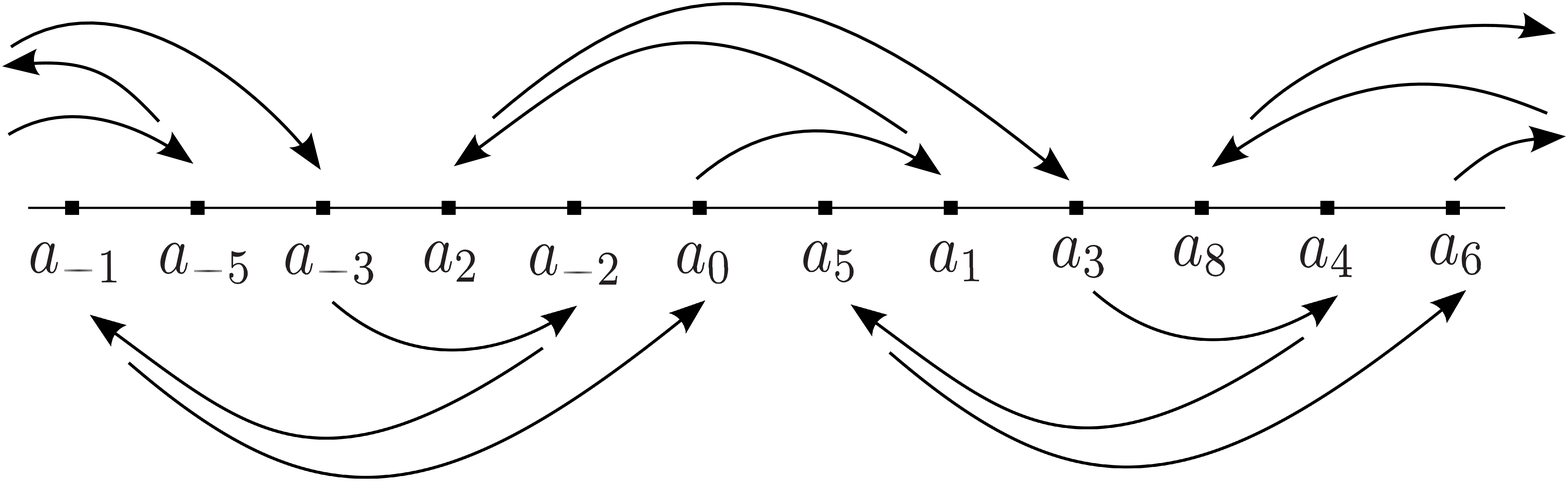} %ファイル名.拡張子を入れる.\ 
       \end{center}
       \caption{\empty}   \label{sbs1}
\end{figure}
In Figure \ref{sbs1}, each arrow represents the orbit of the sequence
$A=\{a_n\}_{n\in \mathbb{Z}}$, that is, each arrow starts from $a_n$ and
gets to $a_{n+1}$ for some $n\in\mathbb{Z}$. Such diagrams as Figure \ref{sbs1}
are useful for our argument, and will be used frequently in later sections.

\begin{deflem} \label{L3}
Let $A=\{a_n\}_{n\in\mathbb{Z}}= \mathbb{Z}$ be a bijective sequence, 
$C\geq 1$, and $\delta \in (0,+\infty]$.
If $A$ is $C-$splittable, then there exists a $K=K(C,\delta)-$quasiconformal mapping
$\tilde{f}:\mathbb{C}\rightarrow \mathbb{C}$ which satisfies the following conditions;
\begin{itemize}
	\item $\widetilde{f}(n)=a_n$ for all $n\in \mathbb{Z}$,
	\item $\widetilde{f}=id$ on $\mathbb{C} \setminus \{z\in\mathbb{C};\ |\im z|< \delta\}$,
\end{itemize}
where $K=K(C,\delta)$ is a constant depending only on $C$ and $\delta$.
\end{deflem}

\begin{proof}
Since $A$ is $C-$splittable, there is a strictly monotone increasing sequence
$\{k_n\}_{n\in\mathbb{Z}}\subset\mathbb{Z}$ such that each interval
$I_n:=[k_n +1,k_{n+1}]_{\mathbb{Z}}$ splits $A$, and satisfies $|I_n|\leq C$.

Since translation $z\mapsto z+\alpha\ (\mathbb{C}\rightarrow\mathbb{C})$
is conformal for any $\alpha\in \mathbb{C}$, 
we may assume $\{a_j\}_{j\in I_n}=[k_n +1,k_{n+1}]_{\mathbb{Z}}=I_n$
for all $n\in\mathbb{Z}$. Namely, the correspondence $j\mapsto a_j
\ (\mathbb{Z}\rightarrow\mathbb{Z})$ splits
into permutations $j\mapsto a_j\ (I_n\rightarrow I_n)\ (n\in \mathbb{Z})$. 
By Lemma \ref{L1}, for each $n\in \mathbb{Z}$, there
exists a $K=K(C,\delta)-$quasiconformal mapping $\widetilde{f}_n:\mathbb{C}\rightarrow
\mathbb{C}$ such that
\begin{itemize}
	\item $\widetilde{f}_n(j)=a_j$ for all $j\in I_n$,
	\item $\widetilde{f}_n=id$ on $\mathbb{C}\setminus \displaystyle
		\left\{z;\ k_n +\frac{1}{2} < \re z < k_{n+1} +\frac{1}{2},\ 
		|\im z|< \delta \right\}$.
\end{itemize}
Let $\widetilde{f}(z):=\widetilde{f}_n(z)$ if 
$z\in \{z;\ k_n+1/2< \re z \leq k_{n+1}+1/2\}.$
Then, clearly $\widetilde{f}$ defines a homeomorphism of $\mathbb{C}$.
Since domains $\{z;\ k_n+1/2< \re z < k_{n+1}+1/2\}
\ (n\in \mathbb{Z})$ are disjoint, the maximal
dilatation of $\widetilde{f}$ is also $K$. Thus $\widetilde{f}$ is a desired mapping.
\end{proof}

\subsection{Three point condition for bijective sequences}\ 

\begin{defdef} 
	Let $A=\{a_n\}_{n\in\mathbb{Z}}=\mathbb{Z}$ be a 
	bijective sequence, and let $\lambda \geq 1$.
	We say that $A$ satisfies the $\lambda-$three point condition if
	for any integers $n<m<k$, it holds that 
	\begin{equation}
		\left|\frac{a_n-a_m}{a_n-a_k}\right|\leq \lambda. \tag{3PC} \label{3PC}
	\end{equation}
	%$\displaystyle \left|\frac{a_n-a_m}{a_n-a_k}\right|\leq \lambda$.
\end{defdef}
Suppose $A$ satisfies the $\lambda-$three point condition. Then for 
any integers $n<m<k$, it holds from the triangle inequality that
\[
	\lambda \geq \left|\frac{a_n-a_m}{a_n-a_k}\right| \geq \left| 1-
	\left|\frac{a_k-a_m}{a_n-a_k}\right|\right|.
\]
Thus, we have a symmetric condition;
\begin{equation}
	\left|\frac{a_k-a_m}{a_k-a_n}\right| \leq \lambda +1. \tag{3PC'} \label{3PC'}
\end{equation}

\begin{defremark} \label{remark1}
	We would like to emphasise that the condition {\rm (\ref{3PC})} holds
	even if $n=m$, and the condition {\rm (\ref{3PC'})} holds even if $m=k$.
	This makes our arguments concise.
\end{defremark}

The conditions (\ref{3PC}) and (\ref{3PC'}) have a simple geometrical meaning;
suppose the orbit starts from a certain point $a_n$ and 
goes far away, say $a_m\ (m>n)$, then the orbit 
$\{a_j\}_{j>m}$ cannot return to a point
near to $a_n$ above a certain rate (see Figure \ref{3pc1}).

\begin{figure}[h]
       \begin{center}
           \includegraphics[width=10cm]{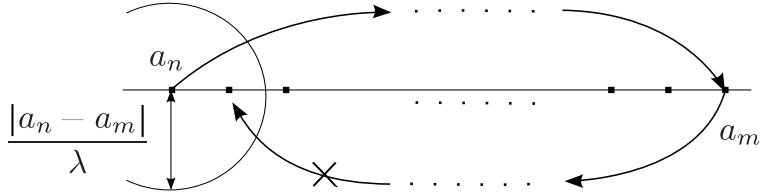} %ファイル名.拡張子を入れる.\ 
       \end{center}
       \caption{:　$\lambda-$three point condition}   \label{3pc1}
\end{figure}

\begin{defprop}\label{P1}
	If a bijective sequence $A=\{a_n\}_{n\in\mathbb{Z}}=\mathbb{Z}$ satisfies
	the $\lambda-$three point condition for some $\lambda \geq 1$, then
	$a_n\rightarrow \pm\infty\ (n\rightarrow \pm \infty)$ or
	$a_n\rightarrow \mp\infty\ (n\rightarrow \pm \infty)$ holds.
\end{defprop}

\begin{proof}
	First, we prove $a_n\rightarrow +\infty$ or $a_n\rightarrow -\infty$ as
	$n\rightarrow +\infty$. To obtain a contradiction, we assume
	$a_n \not \rightarrow \pm \infty$ as $n\rightarrow +\infty$. Further,
	we suppose $a_0=0$ for simplicity.
	
	By the assumption, $\{a_n\}_{n=0}^{+\infty}$ is unbounded from above
	and below. Thus there is an integer $n_1>0$ such that $|a_{n_1}|>
	\lambda+1$ holds, and $a_{n_1}+1 \in \{a_n\}_{n=-\infty}^{-1}$ or 
	$a_{n_1}-1\in \{a_n\}_{n=-\infty}^{-1}$ holds. Indeed if such an integer
	does not exist, $\{a_n\}_{n=0}^{+\infty}$ must contain $\{m\in \mathbb{Z};
	\ |m|>\lambda\}$. This cannot occur since $\{a_n\}_{n=-\infty}^{-1}
	= \mathbb{Z} \setminus \{a_n\}_{n=0}^{+\infty}$ is an infinite subset.
	
	Let $n_{-1}<0$ be an integer which satisfies $a_{n_{-1}}=a_{n_1}+1$ or
	$a_{n_{-1}}=a_{n_1}-1$. Then, by the three point condition,
	We have a contradiction;
	\[
		\lambda \geq \left| \frac{a_{n_{-1}}-a_0}{a_{n_{-1}}-a_{n_1}}\right|
		=|a_{n_{-1}}|>\lambda.
	\]

	Similarly, we can show $a_n\rightarrow +\infty$ or $a_n\rightarrow -\infty$
	as $n\rightarrow -\infty$. Obviously, if $a_n\rightarrow +\infty\ (n
	\rightarrow +\infty)$, then $a_n \not \rightarrow +\infty\ (n\rightarrow -\infty)$.
	Thus we have the claim.
\end{proof}

\subsection{Key Theorem and Extensibility of Quasisymmetric automorphisms}\ 

The following theorem will be proved in Section \ref{proofT1}.

\begin{defthm}\label{T1}
	Let $A=\{a_n\}_{n\in\mathbb{Z}}=\mathbb{Z}$ be bijective, 
	$\lambda \geq 1$, and $\delta\in (0,+\infty]$. If $A$ satisfies the $\lambda-$three
	point condition and $a_n\rightarrow \pm\infty\ (n\rightarrow \pm\infty)$,
	then there exists a $K=K(\lambda,\delta)-$quasiconformal mapping
	$\widetilde{f}:\mathbb{C}\rightarrow \mathbb{C}$ such that
	\begin{itemize}
		\item $\widetilde{f}=id$ on $\mathbb{C}\setminus \{z;\ |\im z|< \delta\}$,
		\item $B=\{b_n:=\widetilde{f}(a_n)\}_{n\in\mathbb{Z}}$ is
			$(2\lambda +3)-$splittable,
	\end{itemize}
	where $K=K(\lambda,\delta)$ is a constant depending only on $\lambda$ and $\delta$.
\end{defthm}

As a direct corollary of Theorem \ref{T1}, we have the following;

\setcounter{alpthm}{1}
\begin{alpthm}{\rm (Extensibility of quasisymmetric automorphisms of $\mathbb{Z}$)}
\label{T3.2}

	For a bijection $f:\mathbb{Z}\rightarrow \mathbb{Z}$, the following
	conditions are quantitatively equivalent;
	\begin{enumerate}
		\item $f$ is $\eta-$quasisymmetric.
		\item $\{a_n:=f(n)\}_{n\in\mathbb{Z}}$ satisfies the $\lambda-$three point condition.
		\item $f$ admits a $K-$quasiconformal extension $\widetilde{f}:
			\mathbb{C}\rightarrow \mathbb{C}$.
	\end{enumerate}
\end{alpthm}

\begin{proof}
	First, $(1)\Rightarrow (2)$ is clear. Indeed, for any integers $n<m<k$, we have
	\[
		\left| \frac{f(n)-f(m)}{f(n)-f(k)} \right|
		\leq \eta \left( \left|\frac{n-m}{n-k}\right|\right) \leq \eta(1).
	\]\vspace{0.5ex}
	
	Next, $(3)\Rightarrow (1)$ is also clear, since $K-$quasiconformal 
	self-homeomorphisms of $\mathbb{C}$
	are $\eta-$quasisymmetric with an $\eta$ depending only on $K$ 
	(thus the restrictions to $\mathbb{Z}$
	are also $\eta-$quasisymmetric with the same $\eta$).\vspace{1ex}
	
	Last, we prove that $(2)\Rightarrow (3)$. By Proposition \ref{P1}, $a_n
	\rightarrow \pm \infty\ (n\rightarrow\pm \infty)$ or $a_n\rightarrow \mp 
	\infty\ (n\rightarrow \pm \infty)$ holds. Since $z\mapsto -z\ (\mathbb{C}
	\rightarrow \mathbb{C})$ is conformal, we may assume the former case holds.
	Then we can apply Theorem \ref{T1}, that is, there exists a 
	$K_1=K_1(\lambda)-$quasiconformal mapping $\widetilde{f}_1:\mathbb{C}
	\rightarrow \mathbb{C}$ such that
	$B=\{b_n:=\widetilde{f}_1(a_n)\}_{n\in\mathbb{Z}}$ is $(2\lambda +3)-$splittable.
	Further, applying Lemma \ref{L3}, we have $K_2=K_2(\lambda)-$quasiconformal
	mapping $\widetilde{f}_2:\mathbb{C}\rightarrow\mathbb{C}$ such that
	$\widetilde{f}_2(n)=b_n=\widetilde{f}_1(a_n)=\widetilde{f}_1 \circ f(n)$.
	Therefore we obtain a desired extension $\widetilde{f}=\widetilde{f}_1^{-1}
	\circ \widetilde{f}_2$.
\end{proof}

Next, we consider quasisymmetric automorphisms of $E:=\{e^n\}_{n\in\mathbb{Z}}$.
In this case, we analogously obtain the following theorem;

\begin{defthm} \label{T3.3}
	For a bijection $f:E=\{e^n\}_{n\in\mathbb{Z}}\rightarrow E$, the
	following conditions are quantitatively equivalent;
	\begin{enumerate}
		\item $f$ is $\eta-$quasisymmetric.
		\item $b_n:=f(e^n) \rightarrow 0$ as $n\rightarrow -\infty$, and 
			$\{b_n\}_{n\in\mathbb{Z}}$ satisfies the $\lambda-$three point
			condition, that is, for any integers $n<m<k$ it holds that
			\[
				\left| \frac{b_n-b_m}{b_n-b_k} \right|
				=\left| \frac{f(e^n)-f(e^m)}{f(e^n)-f(e^k)}\right|
				\leq \lambda.
			\]
		\item $f$ admits a $K-$quasiconformal extension $\widetilde{f}:
			\mathbb{C}\rightarrow\mathbb{C}$.
	\end{enumerate}
\end{defthm}

Assume $f:E\rightarrow E$ is $\eta-$quasisymmetric. Since quasisymmetric
mappings map Cauchy sequences to Cauchy sequences, we have 
$f(e^n)\rightarrow 0\ (n\rightarrow -\infty)$. Further, for any integers $n<m<k$
\[
	\left| \frac{f(e^n)-f(e^m)}{f(e^n)-f(e^k)}\right|
	\leq \eta\left( \left| \frac{e^n-e^m}{e^n-e^k}\right|\right)<\eta(1).
\]
Thus $(1)\Rightarrow (2)$ is valid, and $(3)\Rightarrow(1)$ is clear
for the same reason as the preceding proof.

\begin{defremark}
	In the condition $(2)$, $f(e^n)\rightarrow 0\ (n\rightarrow -\infty)$ is 
	necessary. More precisely, the $\lambda-$three point condition does not imply this
	property. In fact, for $f:e^n\mapsto e^{-n}\ (E\rightarrow E)$, the sequence
	$\{ b_n:=f(e^n)\}_{n\in \mathbb{Z}}$ satisfies the $1-$three point condition,
	but $b_n=f(e^n)\rightarrow +\infty\ (n\rightarrow -\infty)$.
\end{defremark}

Thus we only need to show $(2)\Rightarrow (3)$. To prove this, 
we prepare some lemmas. Let us assume $f:E\rightarrow E$ satisfies the condition $(2)$.\\

Let $a_n:=\log \circ f \circ \exp (n)$ (then $b_n=e^{a_n}$ holds). 
Since $b_n \rightarrow 0\ (n\rightarrow -\infty)$, we have $a_n\rightarrow -\infty\ 
(n\rightarrow -\infty)$. Note that $A:=\{a_n\}_{n\in \mathbb{Z}}=\mathbb{Z}$
is a bijective sequence.

\begin{deflem} \label{L3.4}
	There exists a constant $C_{\lambda}\geq 0$ depending only on $\lambda$,
	such that $a_k-a_{\ell} \leq C_{\lambda}$ holds if $k<\ell$ and $a_{\ell}<a_k$.
\end{deflem}

\begin{proof}
	Assume that integers $k,\ell$ satisfy 
	$k<\ell$ and $a_{\ell}<a_k$. Since $a_n\rightarrow -\infty\ 
	(n\rightarrow -\infty)$, there exists an integer $j<k$ such that
	$a_j<a_{\ell}$. Thus, by the three point condition,
	\begin{eqnarray*}
		\lambda &\geq & \left|\frac{b_j-b_k}{b_j-b_{\ell}}\right| \\
				&=& \left| \frac{e^{a_j}-e^{a_k}}{e^{a_j}-e^{a_{\ell}}}\right|
				=\frac{e^{a_k-a_{\ell}}-e^{-(a_{\ell}-a_j)}}{1-e^{-(a_{\ell}-a_j)}}.
	\end{eqnarray*}
	Since $0<e^{-(a_{\ell}-a_j)}<1$, we have $\lambda \geq e^{a_k-a_{\ell}} -1$.
	Therefore $a_k-a_{\ell}\leq \log (\lambda +1)=:C_{\lambda}$.
\end{proof}
	
\begin{deflem} \label{L3.5}
	$\{a_n\}_{n\in \mathbb{Z}}$ satisfies the $(C_{\lambda}+1)-$three point condition,
	where $C_{\lambda}$ is a constant in Lemma \ref{L3.4}.
\end{deflem}

\begin{proof}
	Let $n<m<k$. If $|a_n-a_k|\geq |a_n-a_m|$, then $|a_n-a_m|/|a_n-a_k|
	\leq 1 \leq C_{\lambda}+1$. Thus we consider the case of $|a_n-a_k|< |a_n-a_m|$.
	
	First, if $a_n>a_m$, then we have
	\[
		\left| \frac{a_n-a_m}{a_n-a_k} \right| \leq a_n-a_m \leq C_{\lambda}
		<1+C_{\lambda}.
	\]
	In this estimation, remark that $a_n$ and $a_k$ are distinct integers,
	that is, $|a_n-a_k|\geq 1$ holds.
	
	Next, if $a_n <a_m$, by $|a_n-a_k|<|a_n-a_m|$ it holds $a_k<a_m$.
	Thus, by Lemma \ref{L3.4} we have $0<a_m-a_k\leq C_{\lambda}$ and
	\[
		\left|\frac{a_n-a_m}{a_n-a_k}\right| \leq
		\frac{|a_n-a_k|+|a_k-a_m|}{|a_n-a_k|} \leq 1+ C_{\lambda}.
	\]
\end{proof}

\begin{defprop} \label{P3.2}
	$(2) \Rightarrow (3)$ in Theorem \ref{T3.3} holds.
\end{defprop}

\begin{proof}
	By Lemma \ref{L3.5} and Proposition \ref{P1}, the sequence $A=
	\{a_n:=\log b_n\}_{n\in\mathbb{Z}}=\mathbb{Z}$ satisfies the
	$(C_{\lambda}+1)-$three point condition and $a_n\rightarrow \pm \infty\ 
	(n\rightarrow \pm \infty)$. Thus by Theorem \ref{T1} and Lemma \ref{L3}, there exists 
	a $K=K(\lambda)-$quasiconformal mapping $\widetilde{G}:\mathbb{C}\rightarrow
	\mathbb{C}$ such that 
	\begin{itemize}
		\item $\widetilde{G}(n)=a_n$,
		\item $\widetilde{G}=id$ on $\mathbb{C}\setminus \{z;\ |\im z|< \pi\}$.
	\end{itemize}
	Define a homeomorphism $\widetilde{F}:\mathbb{C}\rightarrow \mathbb{C}$ by
	\[
		\widetilde{F}(z+2n\pi i):=\widetilde{G}(z)+2n\pi i,
	\]
	for $n\in\mathbb{Z}$, and $z\in \{z;\ -\pi<\im z\leq \pi\}$.
	Clearly, $\widetilde{F}$ is $K=K(\lambda)-$quasiconformal
	and $\widetilde{F}(n)=a_n\ (n\in \mathbb{Z})$.
	Thus, the projection of $\widetilde{F}$ with respect to the
	universal covering $\pi:z\mapsto e^z\ (\mathbb{C}\rightarrow \mathbb{C}^{\ast})$,
	that is, the mapping $\widetilde{f}:\mathbb{C}^{\ast}\rightarrow
	\mathbb{C}^{\ast}$ defined by $\widetilde{f}\circ \pi =\pi \circ \widetilde{F}$
	gives a $K-$quasiconformal extension 
%	$\widetilde{f}:\mathbb{C}^* \rightarrow \mathbb{C}^*$ 
	of $f$. Since $b_n=\widetilde{f}(e^n) \rightarrow 0\ 
	(n\rightarrow -\infty)$,  we
	obtain a desired extension by the removable singularity theorem for 
	quasiconformal mappings.
\end{proof}

\begin{defremark}
	By the construction of $\tilde{f}$, it turns out that
	we can choose the quasiconformal extension in Theorem \ref{T3.3} 
	so that it is identity on the negative real axis.
\end{defremark}

    \medskip
%%%%%%%%%%%%%%%%%%%%%%%%%%%%%%%%%%%%%%%%%%%%%%%%%%%%%%%%%%%%%%%%%%%%%  CHAPTER 5 %%%%%%%%%%%%%%
\section{Proof of Theorem \ref{T1}}
\label{proofT1}

We devote this section to the proof of Theorem \ref{T1}.
The statement of Theorem \ref{T1} is the following;

\begin{empthm}
	Let $A=\{a_n\}_{n\in\mathbb{Z}}=\mathbb{Z}$ be bijective, 
	$\lambda \geq 1$, and $\delta\in (0,+\infty]$. If $A$ satisfies the $\lambda-$three
	point condition and $a_n\rightarrow \pm\infty\ (n\rightarrow \pm\infty)$,
	then there exists a $K=K(\lambda,\delta)-$quasiconformal mapping
	$\widetilde{f}:\mathbb{C}\rightarrow \mathbb{C}$ such that
	\begin{itemize}
		\item $\widetilde{f}=id$ on $\mathbb{C}\setminus \{z;\ |\im z|< \delta\}$,
		\item $B=\{b_n:=\widetilde{f}(a_n)\}_{n\in\mathbb{Z}}$ is
			$(2\lambda +3)-$splittable,
	\end{itemize}
	where $K=K(\lambda,\delta)$ is a constant depending only on $\lambda$ and $\delta$.
\end{empthm}

We would like to start proving this claim. Throughout this section,
we assume $A=\{a_n\}_{n\in \mathbb{Z}}=\mathbb{Z}$ satisfies the $\lambda-$three
point condition $(\lambda \geq 1)$ and 
$a_n\rightarrow \pm \infty\ (n\rightarrow \pm \infty)$.
Further, let $\delta\in (0,+\infty]$.

\subsection*{Step1}
By translation, we may assume $a_0=0$.
Since $a_n\rightarrow \pm \infty\ (n\rightarrow \pm \infty)$,
there uniquely exist integers $k_0$ and $k_1$ such that
\begin{align}
	a_{k_0+1}\geq 0\ \ \ &{\rm and}\ \ \ n<k_0+1\Rightarrow a_n<0,
		\tag{C0} \label{A}\\
	a_{k_1}\leq 0\ \ \ &{\rm and}\ \ \ n>k_1\Rightarrow a_n>0.
		\tag{C1} \label{B}
\end{align}
Remark that since $a_0=0$, it holds that $k_0<0\leq k_1$.

\begin{figure}[htbp]
       \begin{center}
           \includegraphics[width=12cm]{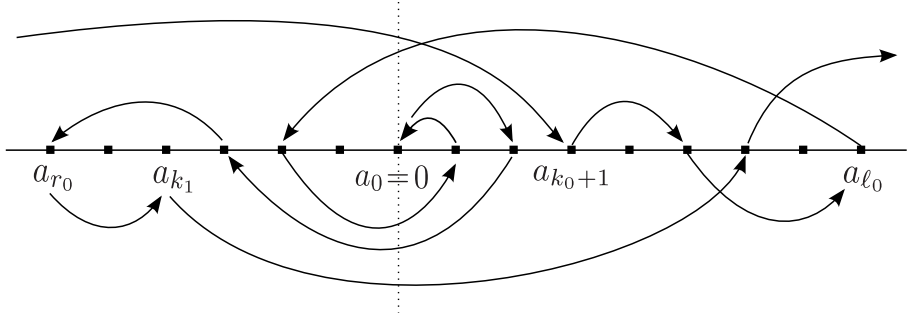} %ファイル名.拡張子を入れる.\ 
       \end{center}
       \caption{:　Orbit of $A=\{a_n\}_{n\in\mathbb{Z}}$}   \label{pT11}
\end{figure}

Let $I_0:=[k_0+1,k_1]_{\mathbb{Z}}$, and let $\ell_0, r_0\in I_0$
satisfy $a_{\ell_0}=\displaystyle \max_{j\in I_0} a_j$, and
$a_{r_0}=\displaystyle \min_{j\in I_0} a_j$ (see Figure \ref{pT11}).
Further, let $\lambda':=\lambda +1$.

\begin{defclaim} \label{claim1}
	$|a_{r_0}|,\ |a_{\ell _0}|\leq \lambda'$ and
	$|I_0|\leq 2\lambda'+1$ hold.
\end{defclaim}

\begin{proof}
	By the three point condition, the following hold for any
	integers $n,m,k\in\mathbb{Z}$
	(see Remark \ref{remark1});
	\begin{eqnarray*} \displaystyle
		n\leq m<k\ \ &\Longrightarrow & \ \ \left|\frac{a_n-a_m}{a_n-a_k}\right|
			\leq \lambda <\lambda' ,\\
		n< m\leq k\ \ &\Longrightarrow & \ \ \left|\frac{a_k-a_m}{a_k-a_n}\right|
			\leq \lambda'.
	\end{eqnarray*}
	Let $j$ be an integer such that $a_j=a_{\ell _0}+1$. Then by the condition
	(\ref{A}), we have $j>k_1(\geq \ell _0)$. Thus,
	\[
		\lambda' \geq \left| \frac{a_{\ell_0}-a_{k_1}}{a_{\ell_0}-a_j}\right|
		=|a_{\ell_0}-a_{k_1}|.
	\]
	Since $a_{\ell_0}\geq 0$ and $a_{k_1}\leq 0$, we have $|a_{\ell_0}|\leq \lambda'$.
	Similarly, we can show $|a_{r_0}|\leq \lambda'$.
	Further, we have $|I_0|\leq a_{\ell_0}-a_{r_0}+1\leq 2\lambda'+1$.
\end{proof}

Next, we sort the interval $[a_{r_0},a_{\ell_0}]_{\mathbb{Z}}$ appropriately
by a global quasiconformal mapping. Let $\{c_n\}_{n=m_0}^{m_1}$ be the
unique sequence such that
\[
	\left\{ 
		\begin{array}{l}
			\{c_n\}_{n=m_0}^{m_1}=\{a_n\}_{n=k_0+1}^{k_1},\vspace{1ex}\\
			a_{r_0}=c_{m_0}<c_{m_0+1}<\cdots
			<c_0=a_0=0<\cdots<c_{m_1}=a_{\ell_0}.
		\end{array}
	\right.
\]
That is, $\{c_n\}_{n=m_0}^{m_1}$ is the ascending sort of $\{a_n\}_{n\in I_0}$
normalized by $c_0=a_0(=0)$. Remark that $m_0$ and $m_1$ are also 
uniquely determined by the above conditions.
Similarly, let $\{d_n\}_{n=p_0}^{-1}$ and $\{d_n\}_{n=1}^{p_1}$ be the
unique sequences such that
\[
	\left\{
		\begin{array}{l}
			\{d_n\}_{n=p_0}^{-1}\cup \{d_n\}_{n=1}^{p_1}=
				[a_{r_0},a_{\ell_0}]_{\mathbb{Z}}\setminus \{a_n\}_{n\in I_0},\vspace{1ex}\\
			d_{p_0}<d_{p_0+1}<\cdots<d_{-1}<0<d_{1}<\cdots<d_{p_1}.
		\end{array}
	\right.
\]
$p_0$ and $p_1$ are also uniquely determined by the conditions $m_0+p_0=a_{r_0}$ and
$m_1+p_1=a_{\ell_0}$. (If $p_0=0$ or $p_1=0$, then we assume the corresponding
sequences are empty.)

By Lemma \ref{L1} and Claim \ref{claim1}, there exists a
$K=K(\lambda,\delta)-$quasiconformal mapping $\widetilde{f}_0:
\mathbb{C}\rightarrow\mathbb{C}$ such that
\begin{enumerate}
	\item $\widetilde{f}_0=id$ on $\mathbb{C}\setminus
		\left\{ z\in\mathbb{C};\ \displaystyle 
		a_{r_0}-\frac{1}{2}<\re z< a_{\ell_0}+\frac{1}{2},\ |\im z|< \delta
		\right\}$,
	\item $\widetilde{f}_0(d_j)=m_0+j\ \ \ (j=-1,-2,\cdots,p_0)$,\vspace{1ex}\\
		$\widetilde{f}_0(c_j)=j\hspace{8ex} (j=m_0,m_0+1,\cdots,m_1)$,\vspace{1ex}\\
		$\widetilde{f}_0(d_j)=m_1+j\ \ \ (j=1,2,\cdots,p_1)$.
\end{enumerate}
Using this mapping, we set $A_0:=\{a_n^0:=\widetilde{f}_0(a_n)\}_{n\in \mathbb{Z}}$
(see Figure \ref{pT12}).
Then $I_0=[k_0+1,k_1]_{\mathbb{Z}}$ splits $A_0$ and $|I_0|\leq 2\lambda'+1=
2\lambda+3$.\vspace{3ex}

\begin{figure}[htbp]
       \begin{center}
           \includegraphics[width=13cm]{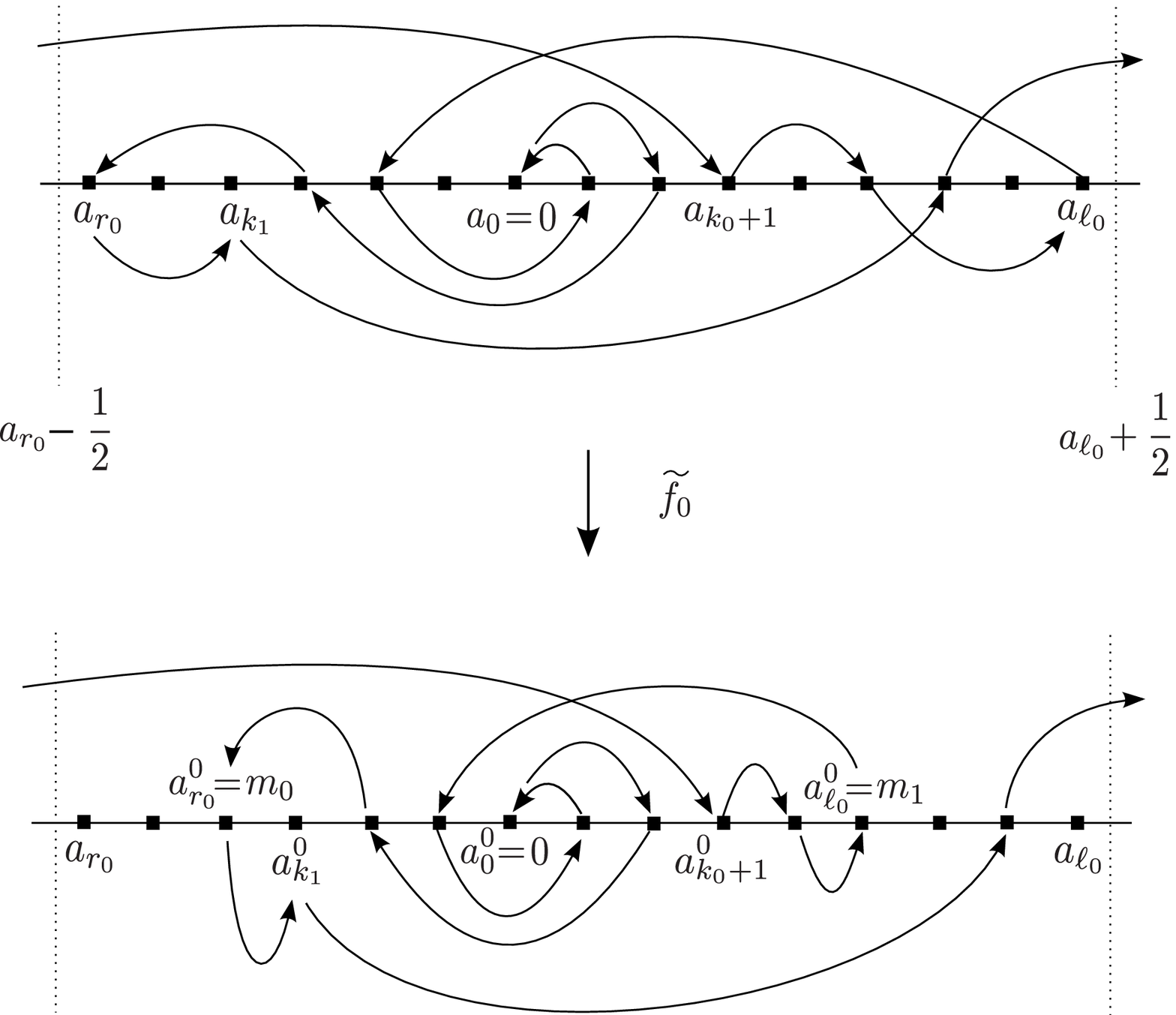} %ファイル名.拡張子を入れる.\ 
       \end{center}
       \caption{:　permutation by $\widetilde{f}_0$}   \label{pT12}
\end{figure}% \vspace{-3ex}

\subsection*{Step2}
By Lemma \ref{L2}, $\{a_n^0\}_{n\in I_0}=[m_0,m_1]_{\mathbb{Z}}$ and
$\{a_n^0\}_{n\in \mathbb{Z}_{>k_1}}=[m_1+1,+\infty)_{\mathbb{Z}}
=\{m_1+1,m_1+2,\cdots\}$.
Let $k_2$ be the maximum integer which satisfies 
$a_{k_2}^0\in [m_1+1,a_{\ell_0}+1]_{\mathbb{Z}}$,
and let $I_1:=[k_1+1,k_2]_{\mathbb{Z}}$. Further 
we let $\ell_1 \in \mathbb{Z}$
be the integer such that $a_{\ell_1}^0=\displaystyle\max_{j\in I_1}a_j^0$.

Remark that since $\widetilde{f}_0=id$ on $\{z;\ \re z>a_{\ell_0}+1/2\}$,
it holds $a_n^0=a_n$ if $a_n^0\geq a_{\ell_0}+1$.
In particular, $a_{\ell_1}^0=a_{\ell_1}$.

\begin{defclaim} \label{claim2}
	$|a_{\ell_1}-(a_{\ell_0}+1)|\leq \lambda'$, and $|I_1|\leq 2\lambda'+1$ hold.
\end{defclaim}

\begin{proof}
	Let $j\in \mathbb{Z}$ satisfy $a_j^0=a_{\ell_1}^0+1=a_{\ell_1}+1$.
	Then $j>k_2(\geq \ell_1)$. Further by the preceding remark, it holds $a_j^0=a_j$.
	By the definition of $\widetilde{f}_0$, we have $a_{\ell_0}+1\geq a_{k_2}^0
	=\widetilde{f}_0(a_{k_2})\geq a_{k_2}$. Therefore,
	\[
		\lambda'\geq \left|\frac{a_{\ell_1}-a_{k_2}}{a_{\ell_1}-a_j}\right|
		\geq|a_{\ell_1}-(a_{\ell_0}+1)|.
	\]
	Moreover by Claim \ref{claim1}, we have
	\begin{eqnarray*}
		|I_1|\leq |[m_1+1,a_{\ell_1}]_{\mathbb{Z}}|&=& a_{\ell_1}-m_1\\
			&\leq &a_{\ell_1}-(a_{\ell_0}+1)+a_{\ell_0}+1 \leq 2\lambda' +1.
	\end{eqnarray*}
\end{proof}

Similarly to Step1, we sort the interval $[a_{\ell_0}+1,a_{\ell_1}]_{\mathbb{Z}}$
appropriately by a global quasiconformal mapping. Let
$\{c_n^1\}_{n=a_{\ell_0}+1}^{m_2}$ be the unique sequence such that
\[
	\left\{
		\begin{array}{l}
			\{c_n^1\}_{n=a_{\ell_0}+1}^{m_2}=[a_{\ell_0}+1,a_{\ell_1}]_{\mathbb{Z}}
			\cap \{a_n^0\}_{n\in I_1},\vspace{1ex}\\
			a_{\ell_0}+1=c_{a_{\ell_0}+1}^1<c_{a_{\ell_0}+2}^1<\cdots<c_{m_2-1}^1
			<c_{m_2}^1=a_{\ell_1},
		\end{array}
	\right.
\]
and let $\{d_n^1\}_{n=1}^{p_2}$ be the unique sequence such that
\[
	\left\{
		\begin{array}{l}
			\{d_n^1\}_{n=1}^{p_2}=[a_{\ell_0}+1,a_{\ell_1}]_{\mathbb{Z}}
			\setminus \{a_n^0\}_{n\in I_1},\vspace{1ex}\\
			d_1^1<d_2^1<\cdots<d_{p_2}^1.
		\end{array}
	\right.
\]
Again, we remark that $p_2$ is automatically determined by the
equation $m_2+p_2=a_{\ell_1}$, and if $p_2=0$, we assume
$\{d_n^1\}_{n=1}^{p_2}$ is empty.
By Lemma \ref{L1} and Claim \ref{claim2}, there exists a $K-$quasiconformal
mapping $\widetilde{f}_1:\mathbb{C}\rightarrow\mathbb{C}$ which satisfies
\begin{enumerate}
	\item $\widetilde{f}_1=id$ on $\mathbb{C}\setminus
		\left\{ z\in\mathbb{C};\ \displaystyle 
		a_{\ell_0}+\frac{1}{2}<\re z< a_{\ell_1}+\frac{1}{2},\ |\im z|< \delta
		\right\}$,
	\item $\widetilde{f}_1(c_j)=j\hspace{8.5ex} (j=a_{\ell_0}+1,
		a_{\ell_0}+2,\cdots,m_2)$,\vspace{1ex}\\
		$\widetilde{f}_1(d_j)=m_2+j\ \ \ (j=1,2,\cdots,p_2)$,
\end{enumerate}
where $K$ is the same constant appeared in the construction of $\widetilde{f}_0$.

\begin{figure}[htbp]
       \begin{center}
           \includegraphics[width=11cm]{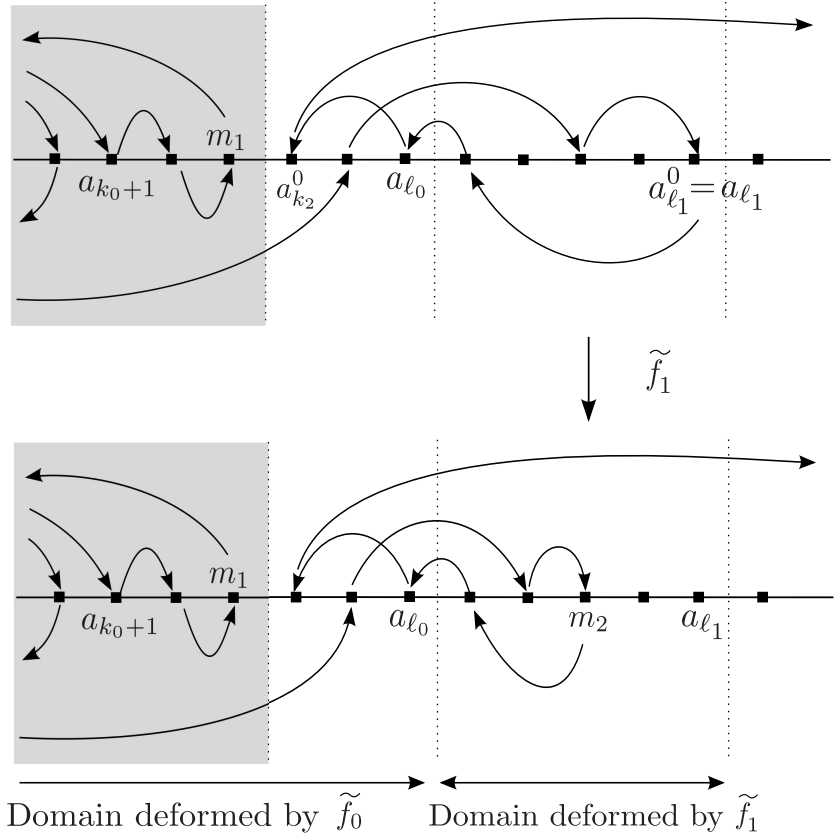} %ファイル名.拡張子を入れる.\ 
       \end{center}
       \caption{:　permutation by $\widetilde{f}_1$}   \label{pT13}
\end{figure}

Let $A_1:=\{a_n^1:=\widetilde{f}_1(a_n^0)\}_{n\in\mathbb{Z}}$. Then intervals
$I_0=[k_0+1,k_1]_{\mathbb{Z}}$ and $I_1=[k_1+1,k_2]_{\mathbb{Z}}$ split $A_1$,
and $|I_0|,\ |I_1|\leq 2\lambda'+1=2\lambda +3$ hold.
Furthermore, since $\widetilde{f}_0$ and $\widetilde{f}_1$ deform disjoint domains, 
the maximal dilatation of $\widetilde{f}_1 \circ \widetilde{f}_0$ 
does not increase. Namely, 
$\widetilde{f}_1 \circ \widetilde{f}_0$ is also $K-$quasiconformal (see Figure 
\ref{pT13}).

\subsection*{Step3}
Applying Step2 repeatedly, for each $m\geq 2$, we can construct
an interval $I_m=[k_m+1,k_{m+1}]_{\mathbb{Z}}$,
a $K-$quasiconformal mapping $\widetilde{f}_m:\mathbb{C}
\rightarrow \mathbb{C}$, and a sequence 
$A_m=\{a_n^m:=\widetilde{f}_m(a_n^{m-1})\}_{n\in\mathbb{Z}}$ such that
\begin{itemize}
	\item intervals $I_0,\ I_1,\cdots,\ I_m$ split $A_m$,\vspace{1ex}
	\item $|I_0|,\ |I_1|,\cdots,\ |I_m|\leq 2\lambda+3$,\vspace{1ex}
	\item $\widetilde{f}_j$ are identity on $\mathbb{C}\setminus
		\{z;\ |\im z|< \delta\}$.
\end{itemize}
Furthermore, by the construction of $\widetilde{f}_j$, the mappings
$\widetilde{f}_j\ (j=0,1,\cdots,m)$ deform disjoint domains. Thus
$\widetilde{f}_m\circ \cdots \circ \widetilde{f}_0$ converges to
a $K-$quasiconformal mapping uniformly on each compact subset 
of $\mathbb{C}$ as $m\rightarrow +\infty$.

Further, we can apply the same argument to the negative direction
of $\{a_n\}_{n\in\mathbb{Z}}$. Consequently we have a desired
$K-$quasiconformal mapping $\widetilde{f}:\mathbb{C}\rightarrow
\mathbb{C}$.
\hfill $\Box$

    \medskip
%%%%%%%%%%%%%%%%%%%%%%%%%%%%%%%%%%%%%%%%%%%%%%%%%%%%%%%%%%%%%%%%%%%%%  CHAPTER 6 %%%%%%%%%%%%%%
\section{Characterization of quasisymmetric images}
\label{images}

In this section, we characterize subsets $E\subset \mathbb{R}$
which are images of some quasisymmetric embeddings
$f:\mathbb{Z}\rightarrow \mathbb{R}$. On the other hand,
the author have characterized images of quasiconformal
mappings as follows; %in \cite[Theorem A]{fujino2}

\begin{defthm} \label{oldthm}
	{\rm (F. 2015 \cite[Theorem A]{fujino2})}
	For a subset $E\subset \mathbb{R}$, the following conditions are 
	quantitatively equivalent.
	\begin{enumerate}
		\item There exists a $K-$quasiconformal mapping $F:\mathbb{C}\rightarrow\mathbb{C}$,
			such that $F(\mathbb{Z})=E$.
		\item $E$ can be written as a monotone increasing sequence $E=\{a_n\}_{n\in \mathbb{Z}}$
			with $a_n\rightarrow\pm \infty\ (n\rightarrow\pm \infty)$, and
			there exists a constant $M\geq 1$ such that the following inequality holds 
			for all $n\in \mathbb{Z}$ and $k\in \mathbb{N}$;
			\[
				\frac{1}{M} \leq \frac{a_{n+k}-a_n}{a_n-a_{n-k}} \leq M.
			\]
	\end{enumerate}
	Further, if $E$ satisfies the second condition, 
	there exists a quasiconformal mapping $F:\mathbb{C}\rightarrow\mathbb{C}$ 
	such that $F(n)=a_n$ for all $n\in \mathbb{Z}$.
\end{defthm}

We will see that the above conditions are desired characterizations.
To see this, we can use almost the same proof as \cite[Theorem A]{fujino2}.
However, we would like to give proofs here for completeness and convenience.
First, we prepare some preliminary lemmas.

\begin{defremark}
	If $E\subset \mathbb{R}$ is an image of a quasisymmetric mapping $f:\mathbb{Z}
	\rightarrow \mathbb{R}$, since quasisymmetric mappings take Cauchy sequences 
	to Cauchy sequences, $E$ must be closed and discrete in $\mathbb{R}$.
\end{defremark}

\begin{deflem} \label{L4.1}
	Let $f:\mathbb{Z}\rightarrow \mathbb{R}$ be an $\eta-$quasisymmetric mapping,
	and let $E:=f(\mathbb{Z})$. Then $\sup E=+\infty$ and $\inf E=-\infty$.
\end{deflem}

\begin{proof}
	To obtain a contradiction, we assume $\inf E>-\infty$. Since $E$ is closed 
	and discrete, we have $\sup E=+\infty$.
	Thus $E$ can be written as a monotone increasing sequence $E=\{a_n\}_{n\in 
	\mathbb{N}}$ with $a_n\rightarrow +\infty$ as $n\rightarrow +\infty$.
	
	Let $g:=f^{-1}:E\rightarrow \mathbb{Z}$. By translation, we may assume
	$g(a_1)=0$. Further, note that $g$ is $\eta'-$quasisymmetric where
	$\eta'(t)=1/\eta^{-1}(1/t)$. Let $\mu :=\eta'(1)$ and consider the set
	\[
		S:=\left\{k\in\mathbb{N};\ g(a_k)=\max_{j=1,2,\cdots,k} g(a_j)\geq \mu \right\}.
	\]
	Since $g:E\rightarrow \mathbb{Z}$ is bijective,
	$S$ consists of infinitely many elements.
	We number $S=\{k_j\}_{j\in \mathbb{N}}$ in ascending order.
	Then the sequence $\{g(a_{k_j})\}_{j\in\mathbb{N}}\subset \mathbb{Z}$
	is monotone increasing.
	On the other hand, there exist infinitely many $n\in \mathbb{N}$ with $g(a_n)<0$.
	Thus we can find $j, \ell \in \mathbb{N}$ such that $k_j < \ell < k_{j+1}$ and
	$g(a_{\ell})<0$. Moreover since $g(a_n)\leq g(a_{k_j})$ for all 
	$n=1,2,\ldots, k_{j+1}-1$, if $g(a_m)=g(a_{k_j})+1$ then $m\geq k_{j+1}$.
	Consequently we confirmed that there exists $k\in S$ and exist 
	$\ell,m \in \mathbb{N}$ such that
	\begin{itemize}
		\item $k<\ell<m$,
		\item $g(a_{\ell})<0$ and $g(a_m)=g(a_k)+1$ (see Figure \ref{img1}).
	\end{itemize}
	\begin{figure}[htbp]
	       \begin{center}
	           \includegraphics[width=12cm]{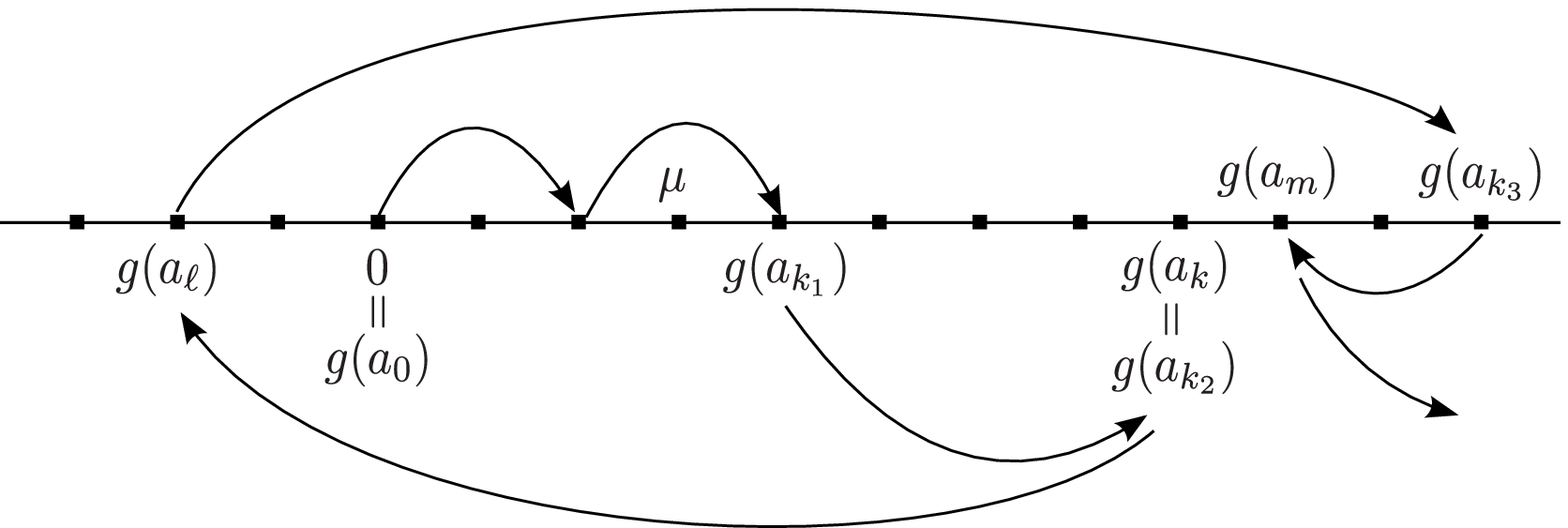} %ファイル名.拡張子を入れる.\ 
	       \end{center}
	       \caption{}   \label{img1}
	\end{figure}
	Therefore, we have a contradiction;
	\begin{eqnarray*}
		\mu > \eta'\left(\left|\frac{a_k-a_{\ell}}{a_k-a_m}\right|\right)
			&\geq &\left| \frac{g(a_k)-g(a_{\ell})}{g(a_k)-g(a_m)}\right|\\
			&=&g(a_k)-g(a_{\ell})>g(a_k)\geq \mu.
	\end{eqnarray*}
\end{proof}

\begin{deflem} \label{L4.2}
	Let $E=\{a_n\}_{n\in\mathbb{Z}}\subset \mathbb{R}$ be a monotone
	increasing sequence with $a_n\rightarrow \pm \infty$ as 
	$n\rightarrow \pm \infty$. If $g:E\rightarrow \mathbb{Z}$
	is an $\eta'-$quasisymmetric bijection, then there exists
	a constant $L\geq 1$ depending only on $\mu:=\eta'(1)$
	which satisfies the following inequality for all 
	$n\in\mathbb{Z}$ and $k\in \mathbb{N}$;
	\[
		\frac{1}{L}< \left|\frac{g(a_{n+k})-g(a_n)}{g(a_n)-g(a_{n-k})}
		\right| < L.
	\]
\end{deflem}

To prove Lemma \ref{L4.2}, first, we prove the following estimation;
\setcounter{defclaim}{0}
\begin{defclaim} \label{subclaim1}
	For any $n\in\mathbb{Z}$, it holds $|g(a_n)-g(a_{n+1})|< 2\mu$.
\end{defclaim}

\begin{proof}
	Since $\mu=\eta'(1)\geq 1$, it suffices to consider the case that
	$|g(a_n)-g(a_{n+1})|\geq 2$. Then we may assume $g(a_{n+1})>g(a_n)$ 
	since the same argument mentioned below can be applied to the case
	$g(a_n)>g(a_{n+1})$.

	Letting $m\in \mathbb{Z}_{\leq n}$ satisfy
	\[
		g(a_m)=\max \left\{ g(a_j);\  j\in \mathbb{Z}_{\leq n} 
		\ \text{such that}\ g(a_n)\leq g(a_j)< g(a_{n+1})\right\}
	\]
	and $\ell \in \mathbb{Z}$ satisfy $g(a_{\ell})=g(a_m)+1$ (
	then $\ell \geq n+1$ by the construction),
	we can construct $m, \ell \in \mathbb{Z}$ which satisfy the following conditions
	(see Figure \ref{img2});
	\begin{enumerate}
		\item $m\leq n$ and $n+1 \leq \ell$,
		\item $g(a_n)\leq g(a_m) < g(a_{\ell})=g(a_m)+1 \leq g(a_{n+1})$.
%		\item $|g(a_m)-g(a_{\ell})|=1$.
	\end{enumerate}
	
%	\ \vspace{-10ex}
	
	\begin{figure}[htbp]
	       \begin{center}
	           \includegraphics[width=12cm]{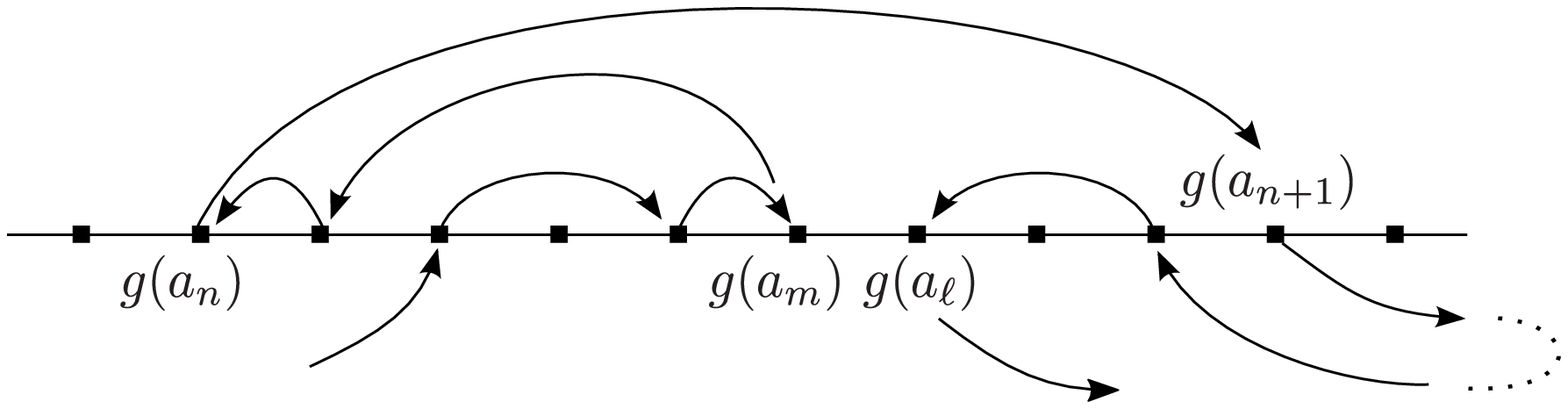} %ファイル名.拡張子を入れる.\ 
	       \end{center}\ \vspace{3ex}
	       \caption{}   \label{img2}
	\end{figure}
	
	First, suppose $g(a_m)-g(a_n)\geq \left( g(a_{n+1})-g(a_n)\right)/2(\geq 1)$. 
	Then $g(a_m),g(a_n),g(a_{\ell})$ are distinct and 
	\begin{eqnarray*}
		\mu &> & \eta'\left(\left|\frac{a_m-a_n}{a_m-a_{\ell}}\right|\right)\\
			&\geq & \left|\frac{g(a_m)-g(a_n)}{g(a_m)-g(a_{\ell})}\right|\\
  			&=& g(a_m)-g(a_n) \geq \frac{g(a_{n+1})-g(a_n)}{2}.
	\end{eqnarray*}
	Thus we have $g(a_{n+1})-g(a_n)< 2\mu$.

	Next, suppose $g(a_m)-g(a_n)< \left( g(a_{n+1})-g(a_n)\right)/2$. 
	Then $g(a_{n+1})-g(a_m)>(g(a_{n+1})-g(a_n))/2$ holds.
	Since $g(a_{n+1})-g(a_n)\geq 2$,
	\[
		g(a_{n+1})-g(a_{\ell})> \frac{g(a_{n+1})-g(a_n)}{2}-1 \geq 0,
	\]
	that is, $\ell \neq n+1$. Therefore $g(a_m),g(a_{n+1}),g(a_{\ell})$ are distinct. 
	Similarly we have $g(a_{n+1})-g(a_n)< 2\mu$.
\end{proof}

\begin{defclaim} \label{subclaim2}
	For any $n\in \mathbb{Z}$ and $k\in\mathbb{N}\ (k\neq 1)$, the following 
	inequality holds;
	\[
		\frac{k-1}{2\mu}< |g(a_n)-g(a_{n+k})|< 2\mu k.
	\]
\end{defclaim}

\begin{proof}
	(\textit{Upper bound}) 
	By the triangle inequality, it immediately follows from Claim \ref{subclaim1} that
	$|g(a_n)-g(a_{n+k})|< 2\mu k$.\\

	(\textit{Lower bound}) Suppose $k\neq 1$. 
	Since the open interval
	\[ 
		\left(g(a_n)-\frac{k-1}{2},\ g(a_n)+\frac{k-1}{2}\right)
	\]
	contains at most $(k-1)$ integer points, 
	there exists an integer $m\in \mathbb{Z}\ (n<m<n+k)$ such that
	\[
		|g(a_n)-g(a_m)|\geq \frac{k-1}{2}.
	\] 
	By the quasisymmetry, we obtain
	\begin{eqnarray*}
		\mu &>& \eta'\left(\left|\frac{a_n-a_m}{a_n-a_{n+k}}\right|\right)\\
			&\geq& \left| \frac{g(a_n)-g(a_m)}{g(a_n)-g(a_{n+k})}\right| 
			\geq \frac{k-1}{2|g(a_n)-g(a_{n+k})|},
	\end{eqnarray*}
	that is, $|g(a_n)-g(a_{n+k})|> (k-1)/2\mu$.
\end{proof}

\begin{defclaim} \label{subclaim3}
	Lemma \ref{L4.2} holds.
\end{defclaim}

\begin{proof}
	If $k\neq 1$, it immediately follows from Claim \ref{subclaim2} that 
	\[
		\frac{1}{L} < \left| \frac{g(a_{n+k})-g(a_n)|}{g(a_n)-g(a_{n-k})} 
		\right| < L
	\]
	for $L=8\mu^2$. Moreover, even if $k=1$, it follows from Claim \ref{subclaim1}
	\[
		8\mu^2>2\mu > \left| \frac{g(a_{n+1})-g(a_n)}{g(a_n)-g(a_{n-1})} 
		\right| > \frac{1}{2\mu} >\frac{1}{8\mu^2}.
	\]
\end{proof}

By the above lemmas, we obtain the following;

\setcounter{alpthm}{2}
\begin{alpthm} \label{T4.1}
	For a subset $E\subset \mathbb{R}$, the following conditions
	are quantitatively equivalent;
	\begin{enumerate}
		\item There exists an $\eta-$quasisymmetric bijection 
			$f:\mathbb{Z}\rightarrow E$.
		\item $E$ can be written as a monotone increasing sequence $E=\{a_n\}_{n\in \mathbb{Z}}$
			with $a_n\rightarrow\pm \infty\ (n\rightarrow\pm \infty)$, and
			there exists a constant $M\geq 1$ such that the following inequality holds 
			for all $n\in \mathbb{Z}$ and $k\in \mathbb{N}$;
			\[
				\frac{1}{M} \leq \frac{a_{n+k}-a_n}{a_n-a_{n-k}} \leq M.
			\]
		\item There exists a $K-$quasiconformal mapping $F:\mathbb{C}\rightarrow\mathbb{C}$,
			such that $F(\mathbb{Z})=E$.
	\end{enumerate}
\end{alpthm}

\begin{proof}
	The equivalence $(2)\Leftrightarrow (3)$ is already confirmed by Theorem \ref{oldthm}
	(see \cite[Theorem A]{fujino2}). Further, for the same reason as the proof of 
	Theorem \ref{T1}, $(3)\Rightarrow (1)$ follows. Thus it suffices to show
	$(1)\Rightarrow (2)$.
	
	Let us assume that there exists an $\eta-$quasisymmetric bijection 
	$f:\mathbb{Z}\rightarrow E$. By Lemma \ref{L4.1}, $E$ can be 
	written as a monotone increasing sequence
	$E=\{a_n\}_{n\in\mathbb{Z}}$ with $a_n\rightarrow \pm \infty$ as $n\rightarrow 
	\pm \infty$ (recall $E$ must be closed and discrete in $\mathbb{R}$). 
	Let $g:=f^{-1}$. Then $g$ is $\eta'-$quasisymmetric where
	$\eta'(t)=1/\eta^{-1}(1/t)$. By Lemma \ref{L4.2}, there exists a constant
	$L\geq 1$ depending only on $\eta'(1)=1/\eta^{-1}(1)$ 
	which satisfies the following inequality
	for any $n\in\mathbb{Z}$ and $k\in\mathbb{N}$;
	\[
		\frac{1}{L}< \left|\frac{g(a_{n+k})-g(a_n)}{g(a_n)-g(a_{n-k})}\right|
		< L.
	\]
	Therefore we obtain
	\[
		\left|\frac{a_{n+k}-a_n}{a_n-a_{n-k}}\right|
		\leq \eta\left( \left|\frac{g(a_{n+k})-g(a_n)}{g(a_n)-g(a_{n-k})}
		\right|\right) < \eta(L).
	\]
	and 
	\[
		\left|\frac{a_{n+k}-a_n}{a_n-a_{n-k}}\right|
		\geq \eta\left( \left|\frac{g(a_{n+k})-g(a_n)}{g(a_n)-g(a_{n-k})}
		\right|^{-1}\right)^{-1} > \frac{1}{\eta(L)}.
	\]
\end{proof}

    \medskip
%%%%%%%%%%%%%%%%%%%%%%%%%%%%%%%%%%%%%%%%%%%%%%%%%%%%%%%%%%%%%%%%%%%%%  CHAPTER 7 %%%%%%%%%%%%%%
\section{Extensibility of quasisymmetric embeddings}
\label{extension}
 
We would like to complete this paper, proving the following theorem;

\setcounter{alpthm}{0}
\begin{alpthm} \label{T5.1}
	Every $\eta-$quasisymmetric embedding $f:\mathbb{Z}\rightarrow 
	\mathbb{R}$ admits a $K=K(\eta)-$quasiconformal extension
	$\widetilde{f}:\mathbb{C}\rightarrow \mathbb{C}$ where
	$K=K(\eta)$ is a constant depending only on $\eta$.
\end{alpthm}

\begin{proof}
	Let $f:\mathbb{Z}\rightarrow \mathbb{R}$ be an $\eta-$quasisymmetric
	embedding, and let $E:=f(\mathbb{Z})$. Then, by Theorem \ref{T4.1},
	there exists a $K'-$quasiconformal mapping $F:\mathbb{C}\rightarrow
	\mathbb{C}$ such that $F(\mathbb{Z})=E$, where $K'$ depends only
	on $\eta$. Since compositions of quasisymmetric mappings are
	also quasisymmetric, $F^{-1}\circ f:\mathbb{Z}\rightarrow\mathbb{Z}$
	becomes an $\eta'-$quasisymmetric automorphism where $\eta'$ depends 
	only on $\eta$.	By Theorem \ref{T3.2}, $F^{-1}\circ f$ admits a
	$K''-$quasiconformal extension $G:\mathbb{C}\rightarrow \mathbb{C}$,
	where $K''$ depends only on $\eta$. 
	Therefore, we obtain a $K=K'K''-$quasiconformal extension 
	$\widetilde{f}=F\circ G:\mathbb{C}\rightarrow \mathbb{C}$ of $f$.
	The proof is completed. 
\end{proof}
    
%ーーーーーーーーーーーーーーーーーーーーーーーーーーーーーーー

%	\medskip
%Acknowledgementーーーーーーーーーーーーーーーーーーーーーー

\subsection*{Acknowledgements}
I am deeply grateful to Professor Takeo Ohsawa for his guidance and helpful advices.
This research is partially supported by Grant-in-Aid for JSPS Fellow 16J02185.

% 参考文献ーーーーーーーーーーーーーーーーーーーーーーーーーー
%\newpage
\bibliographystyle{amsxport} %bibstyle tieice など .bst fileは　/usr/share/texlive/texmf-dist/pbibtex/bst
%\markboth{REFERENCES}{REFERENCES} %{右上}{左上}の表示設定
%\setstretch{1.4} %行間設定 article では出来なかった。
%文字サイズは普通に　{\footnotesize } {\small } などで指定できる。

% \bibliography{database}

\begin{bibdiv}
\begin{biblist}

\bib{alestalo1}{article}{
      author={Alestalo, P.},
      author={V{\"a}is{\"a}l{\"a}, J.},
       title={{Uniform domains of higher order. {III}}},
        date={1997},
        ISSN={0066-1953},
     journal={Ann. Acad. Sci. Fenn. Math.},
      volume={22},
      number={2},
       pages={445\ndash 464},
      review={\MR{1469802}},
}

\bib{beurling1}{article}{
      author={Beurling, A.},
      author={Ahlfors, L.~V.},
       title={{The boundary correspondence under quasiconformal mappings}},
        date={1956},
        ISSN={0001-5962},
     journal={Acta Math.},
      volume={96},
       pages={125\ndash 142},
      review={\MR{0086869 (19,258c)}},
}

\bib{fujino2}{article}{
      author={Fujino, H.},
       title={{The existence of quasiconformal homeomorphism between planes
  with countable marked points}},
        date={2015},
        ISSN={0386-5991},
     journal={Kodai Math. J.},
      volume={38},
      number={3},
       pages={732\ndash 746},
         url={http://dx.doi.org/10.2996/kmj/1446210604},
      review={\MR{3417531}},
}

\bib{heinonen1}{book}{
      author={Heinonen, J.},
       title={{Lectures on analysis on metric spaces}},
      series={{Universitext}},
   publisher={Springer-Verlag, New York},
        date={2001},
        ISBN={0-387-95104-0},
  url={http://dx.doi.org.ejgw.nul.nagoya-u.ac.jp/10.1007/978-1-4613-0131-8},
      review={\MR{1800917}},
}

\bib{trotsenko1}{article}{
      author={Trotsenko, D.~A.},
      author={V{\"a}is{\"a}l{\"a}, J.},
       title={{Upper sets and quasisymmetric maps}},
        date={1999},
        ISSN={1239-629X},
     journal={Ann. Acad. Sci. Fenn. Math.},
      volume={24},
      number={2},
       pages={465\ndash 488},
      review={\MR{1724387}},
}

\bib{vaisala4}{incollection}{
      author={V{\"a}is{\"a}l{\"a}, J.},
       title={{Questions on quasiconformal maps in space}},
   booktitle={{Quasiconformal mappings and analysis ({A}nn {A}rbor, {MI},
  1995)}},
      review={\MR{1488460}},
}

\bib{vellis1}{unpublished}{
      author={Vellis, V.},
       title={{Quasisymmetric extension on the real line}},
        date={2015},
        note={arXiv:1509.06638 [math.MG]},
}

\end{biblist}
\end{bibdiv}


%\nocite{vaisala1}　%引用してないもの

%\mbox{}\newpage　\thispagestyle{empty}　%背表紙(長文用）
%ーーーーーーーーーーーーーーーーーーーーーーーーーーーーーーー

\end{document}